\documentclass[english,11pt]{smfart}
\usepackage{color}
\usepackage{amsthm}
\usepackage{amsmath}
\usepackage{amsfonts}
\usepackage{amstext}
\usepackage[latin1]{inputenc}
\usepackage{amscd}
\usepackage{latexsym}
\usepackage{bm}
\usepackage{amssymb}
\usepackage[all,cmtip]{xy}
\usepackage{euscript}
\usepackage{a4wide}
\usepackage{amsmath,amssymb,graphicx}

\newtheorem{theorem}{Theorem}
\newtheorem{proposition}[theorem]{Proposition}
\newtheorem{lemma}[theorem]{Lemma}
\newtheorem{corollary}[theorem]{Corollary}

\newtheorem{definition}[theorem]{Definition}
\newtheorem{remark}{Remark}

\newcommand{\N}{\mathbb{N}}
\newcommand{\R}{\mathbb{R}}
\newcommand{\TT}{\mathbb{T}}

\newcommand{\Z}{\mathbb{Z}}
\newcommand{\di}{\displaystyle}
\newcommand{\Id}{\mbox{\rm Id}}

\newcommand{\T}{\boldsymbol{T}}

\newcommand{\fonctionsansdef}[3]{\begin{array}[t]{lrcl}#1 :&#2 &\longrightarrow &#3 \end{array}}

\date{}

\begin{document}
\setcounter{tocdepth}{3}
\title[Continuous versus discrete structures I -- Discrete embeddings]{Continuous versus discrete structures I -- Discrete embeddings and ordinary differential equations}
\author{Jacky Cresson$^{1,2}$ and Fr\'ed\'eric Pierret$^1$}
\address{$^1$ SYRTE, Observatoire de Paris, 77 avenue Denfert-Rochereau, 75014 Paris, France}

\begin{abstract}
We define an abstract framework called {\it discrete finite differences embedding} which can be used to obtain discrete analogue of formal functional relations in the spirit of category theory. For ordinary differential equations we exhibit three main discrete associate : the differential, integral or variational discrete embeddings which corresponds to classical numerical scheme including variational integrators.  
\end{abstract}

\maketitle

\noindent {\tiny $^1$ SYRTE, Observatoire de Paris, 77 avenue Denfert-Rochereau, 75014 Paris, France}

\noindent {\tiny $^2$ Laboratoire de Math\'ematiques Appliqu\'ees de Pau, Universit\'e de Pau et des Pays de l'Adour,}

\noindent {\tiny  avenue de l'Universit\'e, BP 1155, 64013 Pau Cedex, France}

\begin{tiny}
\tableofcontents
\end{tiny}

\section{Introduction}\label{section1}

Many different fields of Mathematics and Applied Mathematics deal with the construction of a {\it discrete} analogue of a continuous notion. The aim of this article is to provide an abstract framework to deal with discrete version of ordinary differential equations. In an informal way, we are looking for a "{\it functor}" from the set $\mbox{\bf ODE}$ of ordinary differential equations to the set $\mbox{\bf FDE}$ of finite differences equations 
$$\mbox{\bf ODE} \longrightarrow \mbox{\bf FDE}$$
in a way inspired by {\it category theory} \cite{awodey}. It is probably possible to give a complete categorical formulation of our approach even if we only keep in mind the general strategy.\\

Our objective is not only theoretic. Using this abstract "arrow" that we call {\it discretisation} in the following, we want to precise in particular two points : 
\begin{itemize}
\item On how many objects (sets,mappings, etc) depends a discretisation ? 
\item A discretisation being fixed, to obtain a better understanding of which kind of properties and structures from the continuous objects are preserved in the discrete setting.
\end{itemize}
Of course, many answers already exist in the literature. In particular, the one dealing with Numerical Analysis of differential equations and more specifically {\it Geometric numerical integration} \cite{lubi}. However, as we will see, some discrete analogues defined by these theory are surprisingly far from being satisfying at least from the point of view of structure and formulation, even for well known results as for example discrete analogue of Lagrangian and Hamiltonian systems which is the origin of the present work. In this paper, we make a systematic comparison of our formulation with those obtained in geometric numerical integration as exposed in \cite{lubi} and also some results obtained in the context of variational integrators as exposed in \cite{mars}.\\

In order to define the "arrow" called discretisation, we follow the philosophy of {\it embedding formalisms} developed for example in (\cite{cre1},\cite{cre2},\cite{cgp},\cite{cd}) by defining a general concept of {\it discrete embedding}. An account of this formalism for problems related to discretisation was depicted in the articles \cite{cgp} and \cite{cmt} using {\it time scale calculus}. Here, we provide a more conceptual framework, which allows us to recover many classical methods of numerical analysis in the same setting modifying only some projection and lift maps which are inherent to the definition of a discrete embedding then generalizing our previous work in this direction. Moreover, it allows us to revisit classical definitions and results in numerical analysis. \\

As already pointed out, some part of our approach is related to the {\it time-scale} calculus point of view in particular to results as exposed for example in \cite{cmt} and L. Bourdin (\cite{bourdin1},\cite{bourdin2}). However, the time-scale calculus coincide with standard discretisation procedure of the first order in term of approximation, due to the fact that the time-scale derivatives reduces to the forward and backward Euler derivative by construction. This point is important since in numerical analysis one is interested not only to give a clear connexion between a continuous object and a discrete one but also to obtain a hight order of approximation between the discrete solution and the continuous one. Our formalism naturally extends to such a setting and we give examples of order $2$ and $3$.\\

The paper is organized as follows : In the first part, we define the discrete analogue of continuous objects like functions and differential or integral operators. The main point is to give an explicit connexion between the discrete and continuous case. This is done using some particular mappings that we call discretisation and interpolation in the following. As a consequence, we are able to give without any computations on sums or classical methods of rearranging terms new formulations of classical results (discrete integration by part,discrete fundamental theorem of differential calculus, etc). In a second part, we use the previous formalism to define the discrete embedding of a formal functional. This abstract setting allows us to cover very different objects like differential equations, integral equations or Lagrangian functional. We then describe the three natural discrete way to obtain a discrete analogue of a differential equation : the differential, integral and variational case. Each of these procedure lead to different discrete realization of the same equations.

\part{Discrete versus continuous objects}

\section{Notations}

Let $N\in\N$ and $a,b\in \R$ with $a<b$ and let $h=(b-a)/N$. We denote by $\TT$ the subspace of $\R$ defined by $\TT=h\Z \cap [a,b]$ where $h\Z=\{hz | z\in\R \}$. The elements of $\TT$ are denotes by $t_k$ for $k=0,...,N$. 
The indicator function of a subset $A \subset X$ is denoted by $1_A$ and defined by $1_A  (t)=1$ if $t\in A$ and $0$ otherwise. We adopt the convention that a sum from an integer $i$ to an integer strictly inferior than $i$ is zero.

\subsection{Functional spaces}

We denote by $C([a,b],\mathbb{R}^{d})$ the set of functions $x :[a,b]\rightarrow \R^d$, $d\in \N^{*}$ and by $C^{i}([a,b],\ \R^{d})$ the set of i-th differentiable functions.\\

We then consider $C(\TT,\R)$ as the set of functions with value in $\R$ over $\TT$. We denote by $P_n([a,b],\R)$ the set of continuous functions which are piecewise polynomial of degree $n$ which are defined over subintervals $I^n_{i}=[t_{ni},\ t_{n(i+1)}]$ such that $\bigcup_{i} I^n_{i}=[a,b]$ and where we supposed that $N$ is a multiple of $n$. We denote also by $P_n^{+}([a,b],\R)$ ( resp. $P_n^{-}([a,b],\R)$ ) the set of polynomial functions of degree $n$ over $[t_{ni}, t_{n(i+1)}[$ (resp. \ $]t_{ni},\ t_{n(i+1)}])$ and right (resp. left) continuous. Let $T=T_0, \dots ,T_n \in \R^n$ and $F = F_0 ,\dots ,F_n \in \R^n$, we denote by $l_i (t)=\Pi_{j\not=j} \di\frac{(t-T_j)}{(T_i -T_j)}$, $0\leq i \leq n$. We denote by $P_{L,n}^F (t)=\di\sum_{i=0}^n F_i l_i (t)$ the Lagrange polynomial of degree $n$ associated to $F$ (see \cite{dema},p.21-22).

\section{Continuous versus discrete functions}

\subsection{Discrete functions and discretisation}

Let us begin with a general definition :

\begin{definition}[Discrete function] 
A discrete function is an element $F\in C(\TT ,\R )$.
\end{definition}

By definition, a discrete function is completely characterized by the finite set $F_k =F(t_k)$, $k=0,\dots ,N$.\\

In the following, we illustrate all the discrete notions with a single example given by the following discrete function : $\mathbf{F}=\{ 2,1,3,2,7,5,2 \} \in C(\TT ,\R )$ with $\TT =\{ 0,1,2,3,4,5,6 \}$. 

\begin{figure}[ht!]
\centering
\includegraphics[width=0.4\textwidth,clip]{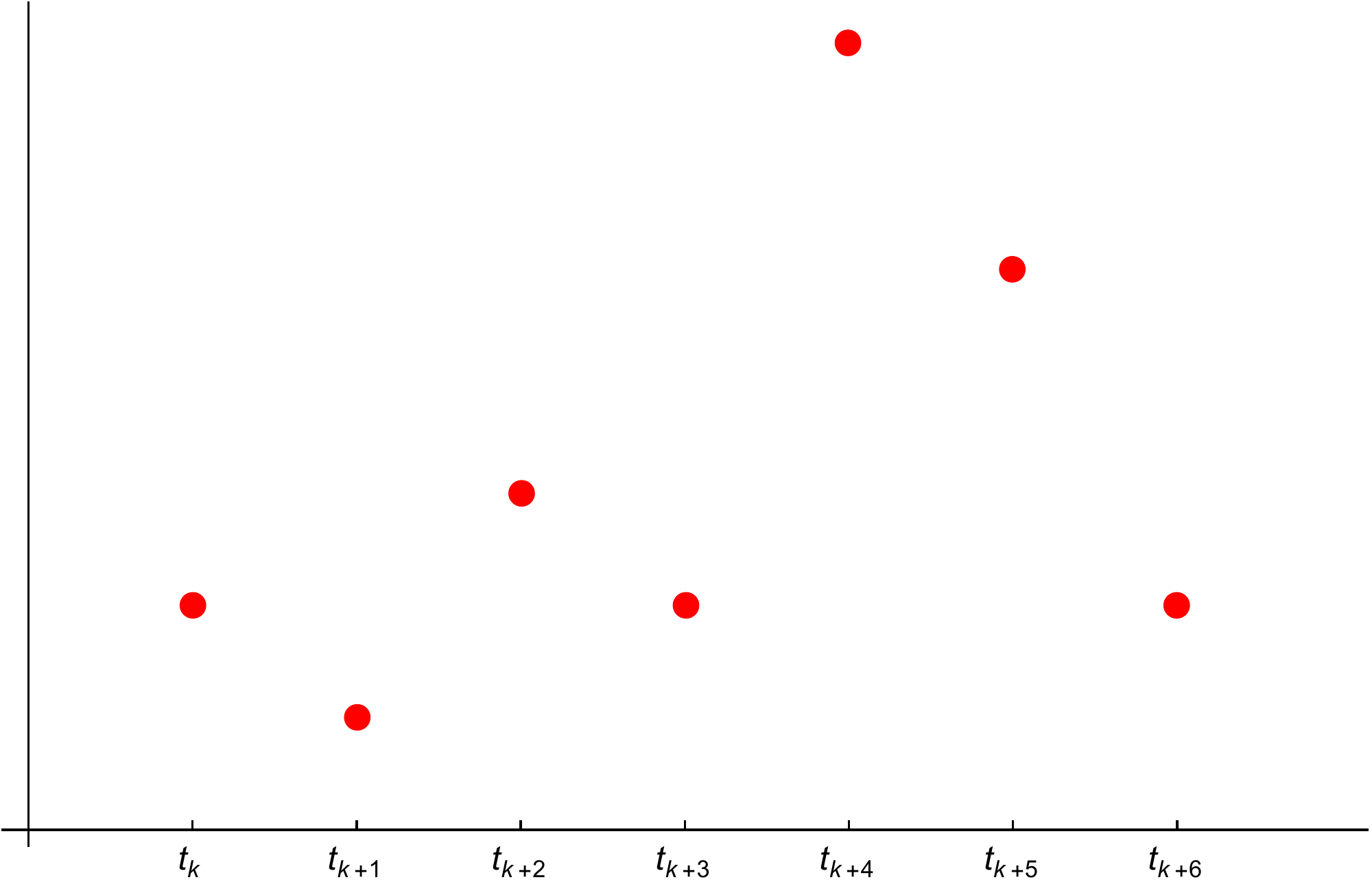}      
\caption{A discrete function}
\label{point}
\end{figure}

\begin{definition}[Discretisation of functions] Let $f\in C([a,b],\R )$, a discretisation of $f$ is a discrete function $F$ such that each $F_k$ can be computed using the value $\{ f_j =f(t_j) \}_{j=0,\dots ,N}$.
\end{definition}

A {\it natural} or {\it canonical discretisation} is of course given by the restriction of $f$ to $\T$.

\begin{definition}[Canonical discretisation]
We denote by $\pi : C([a,b] ,\R ) \rightarrow C(\TT ,\R )$ the mapping define by the restriction of a given function $f\in C([a,b] ,\R )$ to $\TT$, i.e. $\pi (f)=F$ where $F_k =f_k$, $k=,0,\dots ,N$.
\end{definition}

The following section introduces two {\it lift map} from $C(\TT ,\R)$ in various finite vectorial subspaces of $C ([a,b],\R)$. 

\subsection{Interpolation mappings and lifting of discrete functions}

In order to construct discrete analogue of continuous mappings, we have to relate the set of discrete functions to some finite vectorial spaces of $C ([a,b] ,\R )$. \\

\subsubsection{Interpolation mappings}

The notion of {\it interpolation} can be formally defined as follows :

\begin{definition}[Interpolation map]
A map $\iota :C(\TT ,\R) \rightarrow C([a,b] ,\R )$ satisfying $\pi \circ \iota =\Id$, where $\Id$ is the identity of $C(\TT ,\R )$ is called an interpolation map.
\end{definition} 

The name {\it interpolation} comes from the last condition. One can naturally extend a given interpolation mapping for discrete functions with values in $\R^d$ by posing for $F \in C(\TT ,\R^d )$, $F= (F_1 ,\dots ,F_d )$, $F_i \in C(\TT ,\R )$, $\iota (F )=(\iota (F_1 ) ,\dots ,\iota (F_d ) )$. \\

Three classical interpolation mappings will be used in the following (see \cite{dema},Chap.II,$\S$.1). 
 
\begin{definition}[$P_0^{\pm}$--interpolation] We denote by $\iota_0^{\pm} : C(\TT ,\R) \rightarrow P_0^{\pm} ([a,b] ,\R)$ the map defined for all $F\in C(\TT ,\R)$ by
\begin{equation}
\iota_0^+ (F)=\di\sum_{i=0}^{N-1} F_k  1_{[t_k ,t_{k+1}[} 
\ \ \ 
\left ( 
\mbox{\rm resp.} \ \ 
\iota_0^- (F)=\di\sum_{i=1}^{N} F_k  1_{]t_{k-1} ,t_k]} 
\right ) 
.
\end{equation}
\end{definition}

\begin{figure}[ht!]
\centering
\includegraphics[width=0.4\textwidth,clip]{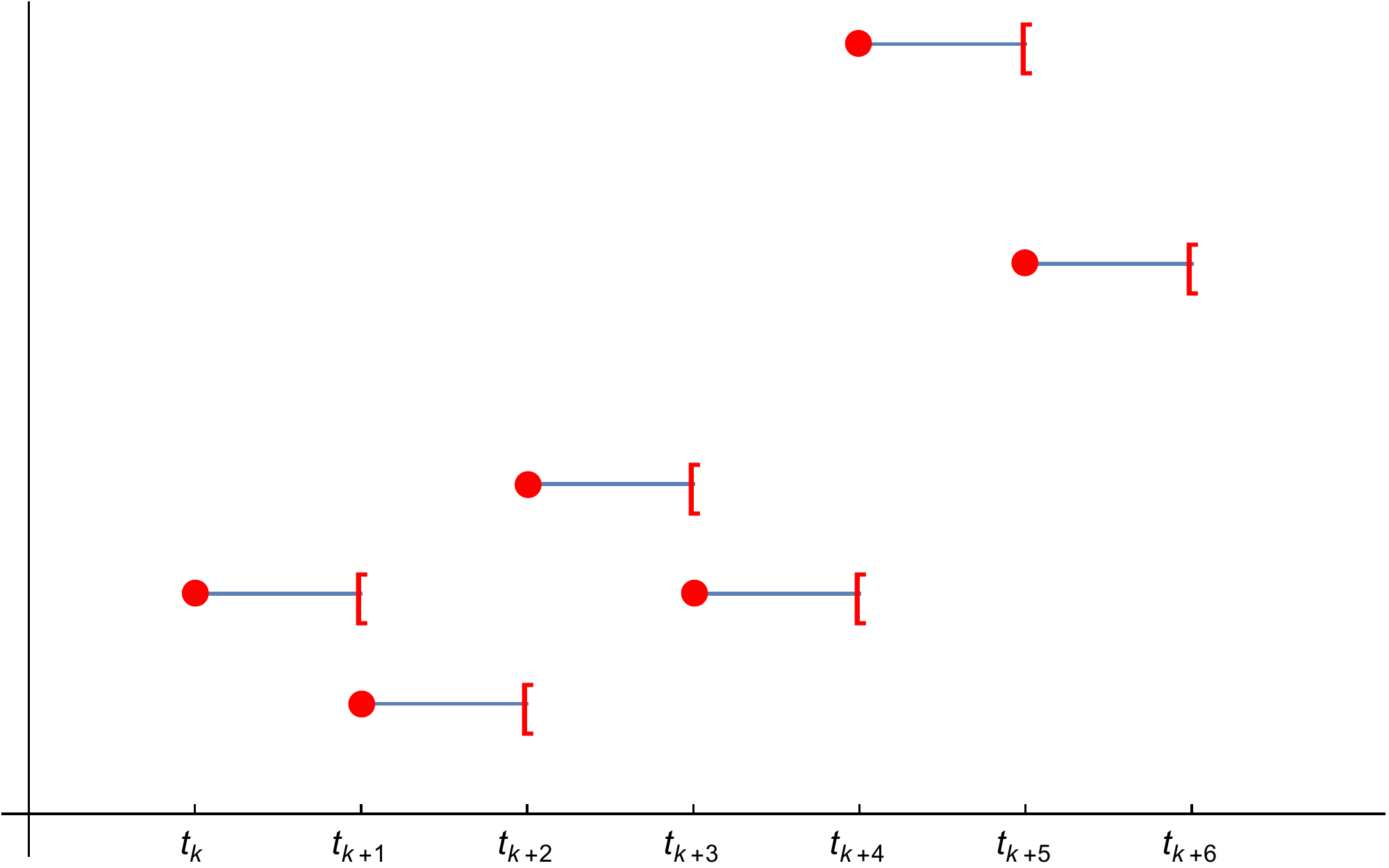}      
\caption{$P_0^+$-interpolation mapping}
\label{p1}
\end{figure}

We need sometimes more regularity.

\begin{definition}[$P^1$--interpolation] We denote by $\iota_1 : C(\TT ,\R ) \rightarrow P^1 ([a,b],\R )$ the map defined for all $F\in C(\TT ,\R )$ by 
\begin{equation}
\iota_1(F)(t)=\sum_{k=0}^{N-1} \left[F_k + \frac{F_{k+1}-F_k}{h}(t-t_k)\right] 1_{[t_k,t_{k+1}]}(t).
\end{equation}
\end{definition}

\begin{figure}[ht!]
\centering
\includegraphics[width=0.4\textwidth,clip]{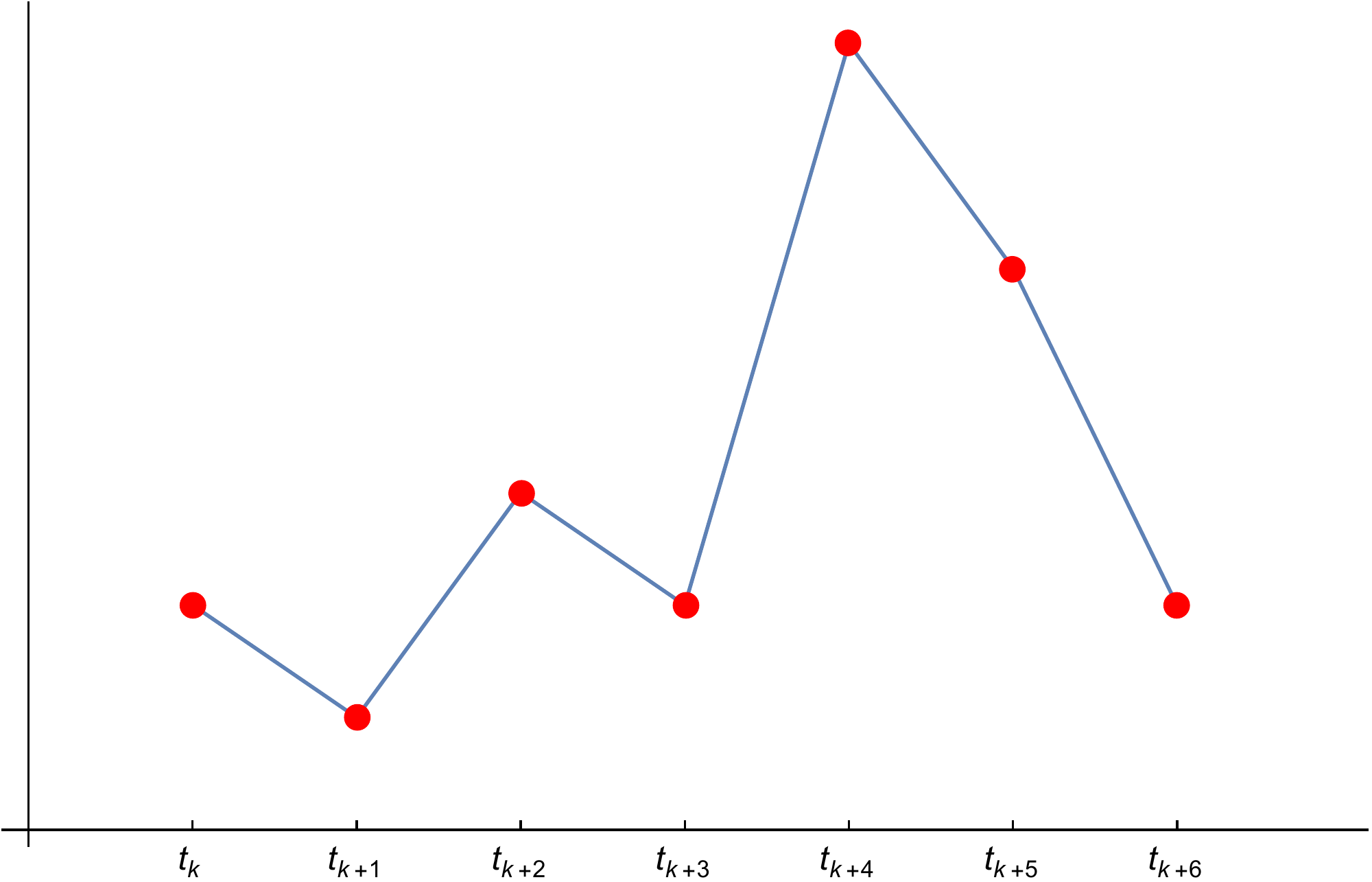}      
\caption{Order one interpolation mapping}
\label{p1}
\end{figure}

This map is of course natural and is classical in numerical analysis. The main problem is that in this case, the derivative is not defined but only left and right derivatives. This point will induce several complications in the connexion between the various discrete operators. 

\subsubsection{Properties of interpolation mappings}

The interpolation mappings satisfy some additional interesting properties :

\begin{lemma}
The restriction of $\iota_0^+ \circ \pi$ and $\iota_0^- \circ \pi$ to $P^0_+ ([a,b])$ and $P^0_- ([a,b])$ respectively is the identity. 
\end{lemma}

\subsection{Algebraic structures}

\subsubsection{Product}

We define a product on $C(\TT,\R )$ which is compatible with the classical product of functions. Let $\star$ be the product defined by 

\[\xymatrixcolsep{5pc} \xymatrixrowsep{5pc}
\xymatrix{ \underset{(\iota_1(F),\iota_1(G))}{P_1([a,b],\R ) \times P_1([a,b],\R )}  \ar[r]^-{ \cdot } & \underset{\iota_1(F) \cdot \iota_1(G)}{P_1([a,b],\R )} \ar[d]^-\pi \\ \underset{(F,G)}{C(\TT,\R ) \times C(\TT,\R )} \ar[u]^-{\iota_1} \ar[r]^-{\star} & \underset{F\star G}{C(\TT,\R )}}
\]

\begin{definition}
The product $\star$ on $C(\TT,\R )$ if defined for all $(F,G)\in C(\TT,\R ) \times C(\TT,\R )$ by $(F\star G)_k =  F_k . G_k $ for $k=0,...,N$.
\end{definition}

Using this definition, we can easily define the product of two vectorial discrete functions :

\begin{definition}
The product $\langle\cdot ,\cdot \rangle_{\star}$ is defined for all $(F,G)\in C(\TT ,\R^d )$ by 
$\langle F , G\rangle_{\star} =\di\sum_{i=1}^d F_i \star G_i$ where $F= (F_1 ,\dots ,F_d )$ and $G= (G_1 ,\dots ,G_d )$.
\end{definition}

We will see that this discrete product can be used to extend the Hilbert structure to discrete functions.

\subsubsection{Duality mappings}

Let $\TT \subset [a,b]$ be given, $\TT =\{ t_0 ,\dots ,t_N \}$. We denote by $\TT_+$ and $\TT_-$ the subset of $\TT$ given by $\TT_+=\{t_0,...,t_{N-1}\}$ (resp. $\TT_-=\{t_1,...,t_{N}\}$). We define $C(\TT_+ ,\R)$ and $C(\TT_- ,\R )$. Many properties of discrete objects come from the date of some {\it natural} mappings between these two functional spaces.

\begin{definition}[Duality mappings] 
We denote by $\sigma :C(\TT_+ ,\R^d ) \rightarrow C(\TT_- ,\R^d )$ (resp. $\rho : C(\TT_- ,\R^d ) \rightarrow C(\TT_+ ,\R^d )$) the map defined by  
\begin{equation}
\sigma (F) (t_i ) =F(t_{i-1} ),\ i=1,\dots ,N\ \ \ 
\left ( 
\mbox{\rm resp.} \ 
\rho (F) (t_i ) =F(t_{i+1} ),\ i=0,\dots ,N-1 
\right ) 
. 
\end{equation}
\end{definition}

\begin{lemma}
The maps $\sigma$ and $\rho$ are bijective and satisfy $\rho \circ \sigma =\sigma \circ \rho =\Id$.
\end{lemma}

The proof follows from the definition of the action of $\sigma$ and $\rho$ on 

\section{Discrete derivatives}

\subsection{Formal construction}

This function is not derivable but we can always define a left and right derivative in the points $t_k$. As a consequence, we can define $\Delta$ and $\nabla$ the two discrete operators corresponding to the left and right derivative denoted by $\frac{d^{+}}{dt}$ and $\frac{d^{-}}{dt}$ by a commutative diagram :

\begin{definition}
The forward (resp. backward) discrete derivative $\Delta$ (resp. $\nabla$) is defined by 
\begin{equation}
\Delta = \pi\circ d^{+} \circ \iota_1 \quad \left( \text{resp.} \quad 
\nabla  = \pi\circ d^{-} \circ \iota_1 \right).
\end{equation}
\end{definition}

This definition corresponds to the following commutative diagram

\[\xymatrixcolsep{5pc} \xymatrixrowsep{5pc}
\xymatrix{ \underset{\iota_1(F)}{P_1([a,b],\R)}\ar[r]^{\di\frac{d^{+}}{dt}} & \underset{\frac{d^{+}}{dt}[ \iota_1(F) ]}{P_0^{+}([a,b],\R)} \ar[d]^\pi \\ \underset{F}{C(\TT,\R)} \ar[u]^{\iota_1} \ar[r]^\Delta & \underset{\Delta F}{C(\TT_+,\R)} }
\]

We deduce easily from the definition that :

\begin{lemma}
The forward (resp. backward) discrete derivative satisfies $\Delta : C(\TT ,\R) \rightarrow C(\TT_+ ,\R)$ (resp. $\nabla : C(\TT ,\R ) \rightarrow C(\TT_- ,\R)$) where $\TT_+=\{t_0,...,t_{N-1}\}$ (resp. $\TT_-=\{t_1,...,t_{N}\}$).
\end{lemma}

\begin{proof}
We have $\frac{d^{+}}{dt} (P^1 ([a,b])=P^0_+ ([a,b[)$ and $\frac{d^{-}}{dt} (P^1 ([a,b])=P^0_- (]a,b])$. We deduce that $\pi (P^0_+ ([a,b[))=C([a,b[\cap \TT ,\R)$ and 
$\pi (P^0_- (]a,b]))=C(]a,b]\cap \TT ,\R)$.
\end{proof}

\begin{figure}[ht!]
\centering
\includegraphics[width=0.4\textwidth,clip]{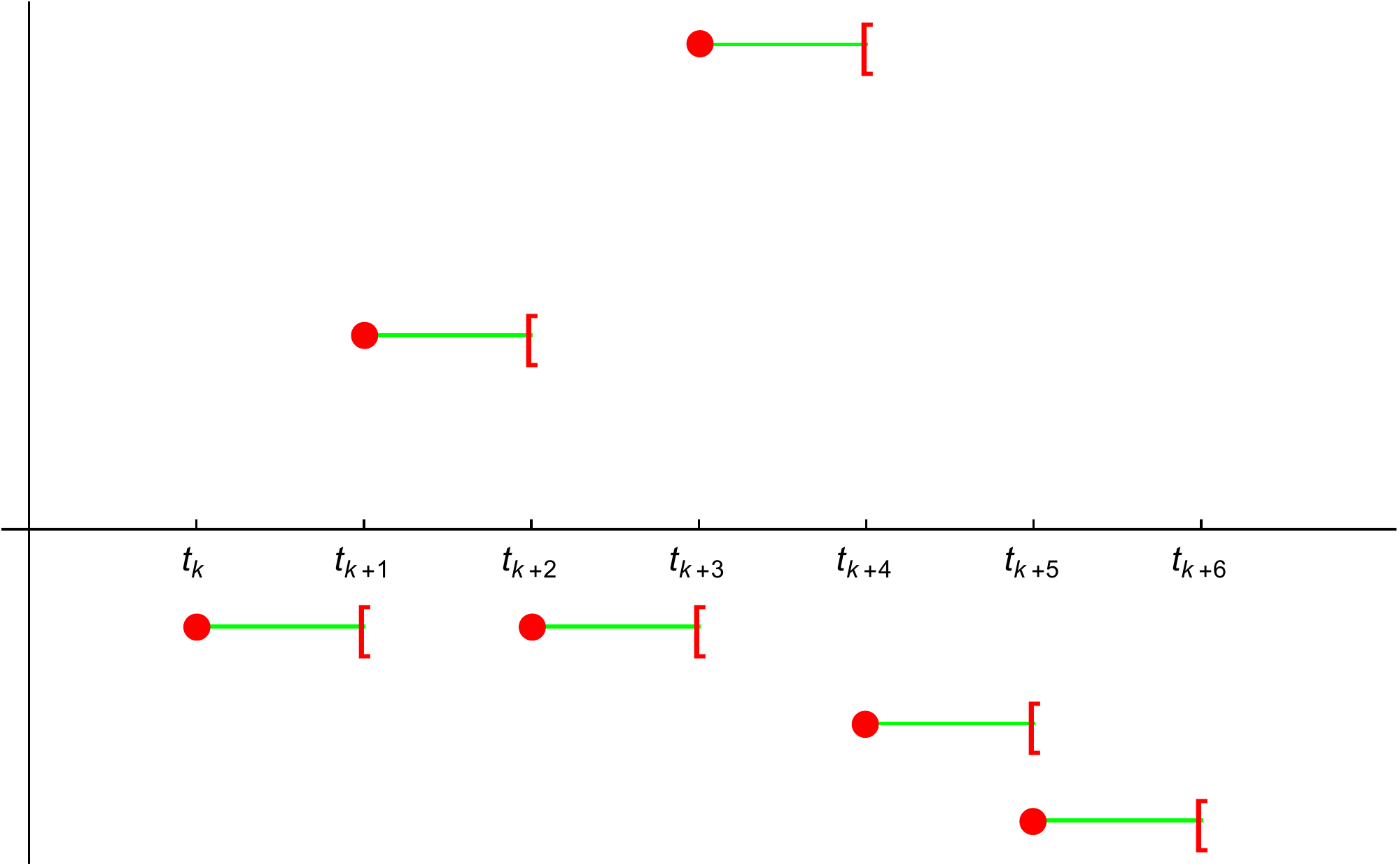}      
\caption{Forward discrete derivative}
\label{dp1}
\end{figure}
\subsection{Explicit form}

It is easy to obtain an explicit form for these two discrete derivatives :

\begin{lemma}
Let $F\in C(\TT ,\R)$, we have $(\Delta F)_k = \frac{F_{k+1}-F_k}{h}$, for $k=0,...,N-1$, and
 $(\nabla F)_k = \frac{F_k-F_{k-1}}{h}$, for $k=1,...,N$. 
\end{lemma}

We recover the classical forward and backward derivatives used in numerical analysis. 

\subsection{Properties of discrete derivatives}

\subsubsection{Duality}

The duality between $C(\TT_+ ,\R )$ and $C(\TT_- ,\R )$ can be used to obtain a duality between the $\Delta$ and $\nabla$ derivative. Precisely, we have :

\begin{lemma}
\label{dual-derive}
Let $F\in C(\TT ,\R )$, then $\nabla (\sigma (F)) =\Delta (F)$ and $\nabla (F) =\Delta (\rho (F))$ over $\TT^{\pm}$.
\end{lemma}

\subsubsection{Kernel of discrete derivatives}

\begin{lemma}
Let $F\in C(\TT,\R)$. We have $\Delta F=0$ (resp. $\nabla F=0$) if and only if $F=\mathbb{F}_0$, where $\mathbb{F}_0$ is defined by $\pi (F_0 Id)$ where $Id$ is the identity function. We will also say that $F$ is constant.
\end{lemma}

The previous Lemma can be used for example to characterize {\it discrete} first integrals of classical numerical scheme (see for example (\cite{bcg},Theorem 12 p.885) and compare with (\cite{lubi}, Theorem 6.7 p.197)).

\subsubsection{Discrete Leibniz formula}

An important algebraic property of the classical derivative is the Leibniz formula. The following theorem gives the discrete version of this formula which mimics exactly the continuous one.

\begin{theorem}[Discrete Leibniz rule]
Let $F,G \in C(\TT,\R)$, we have 
\begin{equation}
\left .
\begin{array}{lll}
\Delta(F\star G) & = & (\Delta F)\star G + \sigma(F) \star (\Delta G)\ \ \mbox{\rm over}\ \TT_+,\\
\nabla(F\star G) & = & (\nabla F)\star G + \rho (F) \star (\Delta G) \ \ \mbox{\rm over}\ \TT_-.
\end{array}
\right .
\end{equation}
\end{theorem}

It is interesting to notice that despite the constant use of the previous formula in many computations concerning numerical scheme, very few exhibit the fact that this is a discrete Leibniz formula. This is due in part to the fact that most of the computations are usually made directly using summations and not in an abstract way. This phenomenon is particularly visible in the derivation of {\it variational integrators} (see \cite{mars}) in the context of the {\it discrete calculus of variations}. 

\section{Discrete antiderivatives}

Using the same procedure as for discrete derivatives, we define discrete antiderivatives.

\subsection{Discrete antiderivatives}

Using the lift mappings we can easily define an antiderivative :

\begin{definition}[Discrete antiderivative]
The discrete $\Delta$-antiderivative (resp. $\nabla$-anti\-derative) denoted by $J_\Delta$ (resp. $J_\nabla$) is defined by
\begin{equation}
J_\Delta =\pi \circ \int_{a}^{t} \circ \, \iota_0^{+} \ \ 
\left ( \mbox{\rm resp.}\ 
J_\nabla =\pi \circ \int_{a}^{t} \circ \, \iota_0^{-}
\right )
.
\end{equation}
\end{definition}

This definition for $J_\Delta$ corresponds to the following diagram 

\[\xymatrixcolsep{5pc} \xymatrixrowsep{5pc}
\xymatrix{ \underset{\iota_0^{+}(F)}{P_0^{+}([a,b],\R)}\ar[r]^{ \di\int_{a}^{t} } & \underset{\di\int_{a}^{t} \iota_0^{+}(F)\, dt}{P_1^{+}([a,b],\R)} \ar[d]^\pi \\ \underset{F}{C(\TT,\R)} \ar[u]^{\iota_0^{+}} \ar[r]^{J_\Delta} & \underset{J_\Delta F}{C(\TT,\R)} }
\]

An analogous diagram is obtained for $J_\nabla$.

\subsection{Explicit form}

An explicit formula for $J_{\Delta}$ (resp. $J_{\nabla}$) is easily obtained : 

\begin{lemma}
Let $F\in C(\TT,\R)$, we have 
$$\left [ J_{\Delta} ( F )\right ]_k  = \di\sum_{i=0}^{k-1}(t_{i+1}-t_i) \, F_i \ \
\left ( 
\mbox{\rm resp}.\ \ 
\left [ J_{\nabla} (F )\right ] _k = \di\sum_{i=1}^k (t_{i}-t_{i-1}) \, F_i 
\right )
$$
for $k=0,...,N$. 
\end{lemma}

\begin{figure}[ht!]
\centering
\includegraphics[width=0.4\textwidth,clip]{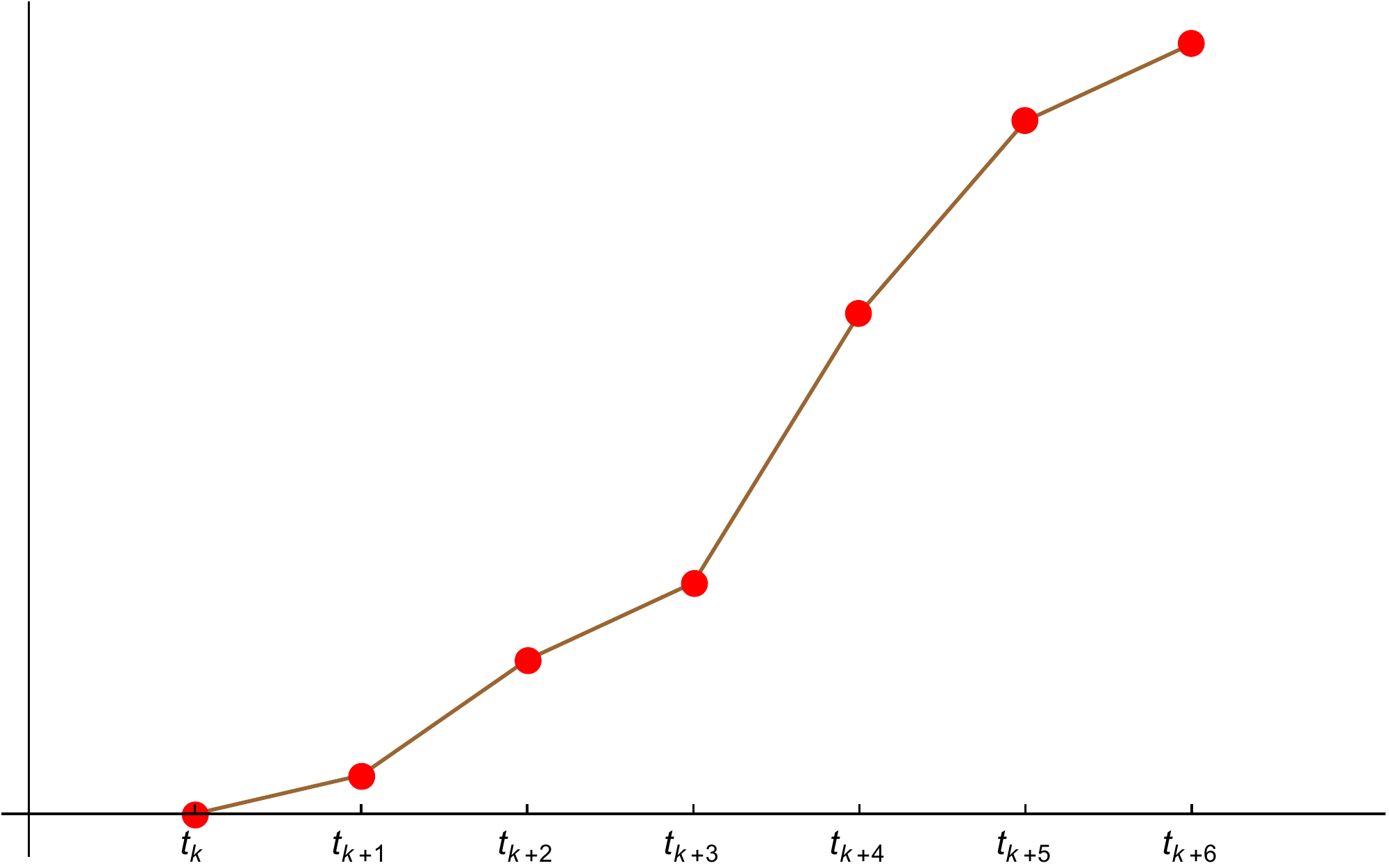}      
\caption{Discrete $\Delta$-integral}
\label{intp1}
\end{figure}

\subsection{Properties of discrete antiderivatives}

\subsubsection{Duality}

The duality between the $\Delta$ and $\nabla$ derivatives induces a corresponding phenomenon for the $\Delta$ and $\nabla$ antiderivatives.

\begin{proposition}
For all $F\in C(\TT,\R)$ we have $J_\Delta \circ \sigma (F) = J_\nabla (F)$ and $J_\nabla \circ \rho (F) = J_\Delta(F)$.
\label{pro_int}
\end{proposition}

\subsubsection{Discrete integration by part}

As in the continuous case, the discrete Leibniz formula for discrete derivatives induces a {\it discrete integration by part} formula for discrete antiderivatives.

\begin{theorem}[Discrete integration by part formula] 
Let $F\in C(\TT ,\R)$ and $G\in C_0 (\TT ,\R )$, we have 
\begin{equation}
\left .
\begin{array}{lll}
\left [ J_{\Delta}(F\star \Delta (G)) \right ]_N & = & - \left [ J_{\Delta} (F \star (\nabla G)) \right ]_N , \\
\left [ J_{\nabla}(F\star \nabla (G)) \right ]_N & = & - \left [ J_{\nabla} (F \star (\Delta G)) \right ]_N .
\end{array}
\right .
\end{equation}
\end{theorem}

\begin{proof}
We only make the proof for $J_{\Delta}$, the case of $J_{\nabla}$ being similar. Using the discrete Leibniz formula for $\Delta$ we have $J_{\Delta}[F\star \Delta G]=J_{\Delta}\left [  \Delta [ \rho (F) \star G ] - \Delta [\rho (F)] \star G  \right ]$ over $\TT_+$. Using Theorem \ref{fondamental} and the fact that $G\in C_0 (\TT ,\R)$, we have $ \left [ J_{\Delta}\circ \Delta(\rho (F)\star G) \right ]_N = \rho(F)_N \cdot G_N -\rho (F)_0 G_0 =0$ . As $\Delta [\rho (F)] =\nabla (F)$ over $\TT^{\pm}$ by Lemma \ref{dual-derive}, this concludes the proof.
\end{proof}

The previous result is never stated as below in classical Textbooks about finite differences. This is due to the fact that the operators are not usually used to write such a formula but directly on summations formulae only speaking of rearranging the terms of the sum (see (\cite{mars},p.363) for a typical example).

\subsubsection{Scalar product}

Let $f,g \in C([a,b],\R^d)$. The classical scalar product on $L^2$ functions denoted by $<\cdot,\cdot>_{L^2}$ is defined by  $<f,g>_{L^2}=\di\int_{a}^{b} <f(t),g(t)> \ dt$. Using the {\it discrete product} and the lift mappings, we can transport the $L^2$ scalar product over discrete functions. 

\begin{definition}[Discrete scalar product]
We call forward (resp. backward) discrete scalar product the bilinear mapping defined for all $F,G \in C(\TT ,\R^d )$ by 
\begin{equation}
\langle F , G\rangle_+ =\di J_{\Delta} [\langle F, G \rangle_{\star}  ]_N \ \ 
\left ( 
\mbox{\rm resp.}\ \  
\langle F , G\rangle_- =\di J_{\nabla} [\langle F, G \rangle_{\star}  ]_N \,
\right ) .
\end{equation}
\end{definition}

\begin{remark}
The discrete forward (resp. backward) discrete scalar product is {\it degenerate}. Indeed, the equation $\langle F,F \rangle _+ =0$ (resp. $\langle F,F \rangle _- =0$ ) induces only
\begin{equation}
F_k = 0 \quad \text{for} \ k=0,...,N-1, \quad  \left( \text{resp.} \quad F_k = 0 \quad \text{for} \ k=1,...,N \right).
\end{equation}
\end{remark}

Using the definition of discrete antiderivatives, we prove the following lemma :

\begin{lemma}
The following diagram commutes 
\[\xymatrixcolsep{5pc} \xymatrixrowsep{5pc}
\xymatrix{ \underset{(\iota_0^{+}(F), \iota_0^{+}(G))}{P_0^{+}([a,b],\R^d ) \times P_0^{+}([a,b],\R^d)}\ar[r]^{ \langle \cdot ,\cdot \rangle } & \underset{\langle \iota_0^{+}(F) ,\iota_0^+ (G) \rangle }{\R)} \ar[d]^{\mbox{\rm Id}} \\ \underset{(F,G)}{C(\TT,\R^d )\times C(\TT ,\R^d )} \ar[u]^{\iota_0^{+}} \ar[r]^{\langle \cdot ,\cdot  \rangle_{\Delta} } & \underset{\langle  F ,G \rangle _{\Delta} }{\R} }
\]

\end{lemma}

The same results occur with $\nabla$ instead of $\Delta$ and $\iota_0^-$ instead of $\iota_0^+$. \\

In other word, the discrete scalar product is just the usual $L^2$ scalar product with $L^2$ functions replaces by discrete functions, the usual multiplication by its discrete analogue and the classical antiderivative by its discrete pendant. This phenomenon is in fact general. We will formalize this property in the next part using discrete embeddings formalisms.

\section{Discrete classical results}

In this Section, we derive two discrete analogue of classical results in Analysis : the fundamental theorem of differential calculus and the Dubois-Raymond lemma.

\subsection{Discrete fundamental theorem of differential calculus}

As we are looking for the transfer of algebro-analytic properties of derivatives and antideratives a natural question is up to which extent the {\it fundamental theorem of differential calculus} (see \cite{hairer},Theorem 6.13 p.239) is preserved ? A very nice feature of our derivation is that this theorem is the following discrete analogue of this result:

\begin{theorem}[Fundamental theorem of the discrete differential calculus]
\label{fondamental}
For all $F \in C(\TT,\R)$ we have 
\begin{equation}
\left .
\begin{array}{lll}
J_\Delta \circ \Delta (F) & = & F - \mathbb{F}_0 \quad \text{and} \quad \Delta\circ J_\Delta F = F,\\
J_\nabla \circ \nabla (F) & = & F - \mathbb{F}_0 \quad \text{and} \quad \nabla\circ J_\nabla F = F.
\end{array}
\right .
\end{equation}
\end{theorem}

\begin{proof}
We make the proof only for $\Delta$ and $J_{\Delta}$. The proof for $\nabla$ and $J_{\nabla}$ being similar. \\

As $d^{+}\circ \iota_1(F) \in P_0^+([a,b],\R)$, we deduce that ${\iota_0^{+}} \circ \pi ( d^{+}\circ \iota_1(F) ) =d^{+}\circ \iota_1(F)$. As a consequence, we obtain
\begin{equation}
\left .
\begin{array}{ll}
J_\Delta\circ\Delta(F)&=\pi \circ \di{\int_a^t} \circ {\iota_0^{+}} \circ \pi \circ d^{+}\circ \iota_1(F) =\pi\circ{\int_{a}^{t}}d^{+}(\iota_1(F)) ,\\
&=\pi\left(\iota_1(F)(t)-\iota_1(F)(0)\right)  =F-\mathbb{F}_0.
\end{array}
\right .
\end{equation}
Second we remark that over $P_1([a,b],\R)$ we have $\iota_1 \circ \pi = Id$. As $\di\int_{a}^{t} \iota_0^{+}(F) \in P_1([a,b],\R)$ then by definition we have
\begin{equation}
\left .
\begin{array}{ll}
\Delta\circ J_\Delta(F)&=\pi \circ d^{+} \circ \iota_1 \circ J_\Delta (F) =\pi\circ d^{+} \circ \iota_1 \circ\pi \circ \di\int_{a}^{t} \iota_0^{+}(F)  ,\\
&=\pi\left(d^{+}\circ\di\int_{a}^{t} \iota_0^{+}(F)\right)=\pi\left(\iota_0^{+}(F)(t)\right) =F.
\end{array}
\right .
\end{equation}

This concludes the proof.
\end{proof}

\subsection{Discrete Dubois-Raymond lemma}

The discrete version of the Dubois-Raymond lemma is valid. We first introduce the set $C_0(\TT,\R)\subset C(\TT,\R)$ defined by

\begin{equation}
C_0(\TT,\R)=\{G\in C(\TT,\R),\ G_{0}=G_{N}=0\}.
\end{equation}

\begin{lemma}
Let $F\in C(\TT,\R)$ such $\left [ J_{\Delta}(F\star G) \right ]_N =0$ for all $G\in C_0(\TT,\R)$ then $F_{k}=0$ for $k=1,\ N-1$.
\end{lemma}

\begin{proof}
Let $G\in C_0(\TT,\R)$. We choose $G$ such that $G_k=F_k$ for $1\le k \le N-1$. Hence, we obtain
$\left [ J_{\Delta}(F\star G) \right ]_N =\di\sum_{k=1}^{N-1}F_{k}^{2}h=0$. We deduce $F_k =0$ for $k=1 ,\dots ,N-1$. This concludes the proof.
\end{proof}

A stronger result can be derived when $F$ is replaced by $\Delta F$ :

\begin{lemma}
Let $F\in C(\TT,\R)$ such that $\left [ J_{\Delta}((\Delta F)\star G) \right ] _N =0$ for all $G\in C(\TT_+ ,\R) $ then $\Delta F=0$.
\end{lemma}

\begin{proof}
The proof follows from a simple computation. As $\left [ J_{\Delta}(\Delta F\star G) \right ]_N =\di\sum_{k=0}^{N-1}\frac{F_{k+1}-F_{k}}{h}G_{k}h$, we obtain, taking $G=\Delta F\in C(\TT_+ ,\R)$ that $F_{k+1}-F_{k}=0$ for all $k=0$, , $N-1$, so that $F=\mathbb{F}_0$. As a consequence, $F$ is a constant for $\Delta$ which concludes the proof.
\end{proof}

Similar results can be obtained for the $\nabla$-derivative.

\part{Discrete embedding formalisms}

We use the previous abstract approach to finite differences in order to provide a formal definition of a discrete analogue for differential equations, functional and other objects which can be defined using integrals, derivatives and functions. We first define a general abstract procedure and give three natural discrete generalisation of a differential equations. This point of view allows us to more precisely determine on which assumptions a discretisation can be constructed. 

\section{Abstract discrete embedding}

In this Section, we consider a general {\it formal functional} made of symbols $d/dt$ and $\int$ acting on a given set of functions $x$,$y$, etc. We denote such a formal functional by 
\begin{equation}
\label{formal}
F(t,x,y; d/dt ,\int ) ,
\end{equation}
as long as the expression defined by $F$ is well defined. \\

A {\it formal relation} will be the data of a formal functional satisfying 
\begin{equation}
\label{relation}
F(t,x,y; d/dt ,\int ) =0.
\end{equation}
A classical formal relation is given by a first order differential equation $dx/dt -f(x,t)=0$. \\

We can now define what is a discrete embedding of an abstract functional or relation.

\begin{definition}[Abstract discrete embedding] Let $\mathcal{E}$ be a given discrete embedding and $F$ be a formal functional defined by (\ref{formal}). The discrete embedding of $F$ denoted by $F_{\mathcal{E}}$ is defined for all $X\in \mathcal{E}$ by
\begin{equation}
F_{\mathcal{E}} (T,X) = F(T,X ; (d/dt)_{\mathcal{E}} , \int_{\mathcal{E}} ) ,
\label{discreteformal}
\end{equation}
where $\mathcal{E}$, $(d/dt)_{\mathcal{E}}$ and $\int_{\mathcal{E}}$ are the discrete functional space, the discrete derivative and discrete antiderivative which are fixed.
\label{discretembed}
\end{definition}

\begin{remark}
The previous procedure deals with the minimal objects used to write a given differential relation. In many situations as for example geometry, a differential relation put in evidence some particular operators, like the Laplacian, from which a differential equations is constructed. In such a context, one can be conducted to define a discrete embedding directly focusing on the given operator and its algebraic or geometric properties. The discrete embedding then follows the same lines as in Definition \ref{discretembed} but the "functorial" property that we are looking for is destroyed. We refer to \cite{CL2014} for more details. 
\end{remark}

\section{Application to ordinary differential equations : the three forms}

In this Section, we apply the previous formalism for ordinary differential equations. Using different representations of a given differential equation (differential or integral form, variational), we obtain discrete analogues which do not give always the same object. We have then multiplicity of {\it discrete realisations} of a given differential equations. This problem is in fact relevant in all embedding formalisms and leads to the {\it coherence problem} asking for conditions under which such representations coincide. 

\subsection{Discrete differential embedding}

Let $ x\in \R^d$, we consider the ordinary differential equations

\begin{equation}
\frac{dx}{dt}=f(t,x). \label{eqdiff}
\end{equation}

Using the finite differences embedding, the discrete $\Delta$ version of this equation is

\begin{equation}
\Delta X =f(T,X)\ , X \in C(\TT,\R^d)), T \in C(\TT,\R)
\end{equation}

As $\Delta X$ is defined on $\mathbb{T}_\Delta$, we obtain for each $i=0,...,N-1$

\begin{equation}
\frac{X_{i+1}-X_{i}}{h}=f(T_{i}, X_{i})\ ,
\end{equation}

where $T_{i}=a+i \frac{b-a}{h}$.

Also we have the discrete $\nabla$ version of this equation which is

\begin{equation}
\nabla X =f(T,X)\ , X \in C(\TT,\R^d)), T \in C(\TT,\R).
\end{equation}

As $\nabla X$ is defined on $\mathbb{T}_\nabla$, we obtain for each $i=1,...,N$

\begin{equation}
\frac{X_{i}-X_{i-1}}{h}=f(T_{i}, X_{i})\ ,
\end{equation}

\subsection{Discrete integral embedding}

The integral formulation of the previous ordinary differential equation is given by
\begin{equation}
x(t)\ =x(0)+\int_{0}^{t}f(s, x(s))ds. \label{eqint}
\end{equation}
The forward discrete embedding of this equation is then given by
\begin{equation}
X=\mathbb{X}_0+J_{\Delta}(f(T,X)),
\end{equation}
which we call the {\it forward discrete integral embedding}. By definition, this equation is equivalent to $X_{i}=X_0+h\sum_{j=0}^{i-1}f(T_j,X_j)$, $i=1,...,N$. As a consequence, we obtain $X_{i+1}-X_i=hf(T_i, X_i)$, $i=0,...,N-1$, which is the classical {\it one step forward Euler scheme} in Numerical Analysis (\cite{dema},Chap.V,$\S$.2.3).\\

In the same way one can obtain the {\it backward discrete integral embedding} of this equation which is then given by
\begin{equation}
X=\mathbb{X}_0 +J_{\nabla}(f(T,X)).
\end{equation}
By definition, this equation is equivalent to $X_{i}=X_0+h\sum_{j=0}^{i-1}f(T_j,X_j)$, $i=1,...,N$. Easy computations lead to $X_{i}-X_{i-1}=hf(T_i, X_i)$, $i=0,...,N-1$, which is the classical {\it one step backward Euler scheme} in Numerical Analysis.

\begin{remark}
The finite differences differential embedding of the equation is equivalent to the differential case. As a consequence, we see that in this simple case, we have coherence between the two discrete versions of the equation.
\end{remark}

\subsection{Discrete variational embedding}

In this Section, we define the discrete variational embedding of a second order differential equation which is Lagrangian. Our derivation is compared with the classical work of J.E. Marsden and M. West \cite{mars} (see also \cite{lubi}) about the discrete calculus of variation and variational integrators, in order to explain the interest of our abstract framework. 

\subsubsection{Lagrangian systems}

In this section, we recall classical definitions concerning Lagrangian systems.

\begin{definition}
A Lagrangian functional is an application denoted by $\mathcal{L} : C^{2}([a,b],\R^{d} ) \rightarrow \R$ and defined by $\mathcal{L}{q}=\di\int_{a}^{b} L(q(t),\dot{q}(t),\ t)dt$, where $L : [a,\ b]\times \R^{d}\times \R^{d} \rightarrow \R$ called Lagrangian and denoted by $L(t,x,v)$.
\end{definition}

The functional $\mathcal{L}$ is also called action functional.

Let $\mathcal{L}$ be a Lagrangian functional, we denote by $D\mathcal{L}(q)(w)$ the Fr\'{e}chet derivative of $\mathcal{L}$ in $q$ along the direction $w$ in $C^{2}([a,\ b],\ \mathbb{R}^{d})$ , i.e.
\begin{equation}
D\mathcal{L}(q)(w)=\lim_{\epsilon \rightarrow 0}\frac{\mathcal{L}(q+\epsilon w)-\mathcal{L}(q)}{\epsilon}.
\end{equation}

An {\it extremal} (or {\it critical point}) of a Lagrangian functional $\mathcal{L}$ is a trajectory $q$ such that $D\mathcal{L}(q)(w)=0$ for any {\it variations} $w$ (i.e. $w \in C^1_0([a,b],\R^{d}):=\{ w \in C^1([a,b],\R^d), \; w(a) = w(b) = 0\}$) .

\subsubsection{Discrete Lagrangian functionals}
The discrete calculus of variations is defined over discrete Lagrangian functionals which are obtained using the discrete embedding that we have fixed. We introduce $\TT^{\pm} :=\TT \backslash \{t_0,t_N\}$.

\begin{definition}
Let $L$ be an admissible Lagrangian function and $\mathcal{L}$ the associated functional. The discrete forward Lagrangian functional $\mathcal{L}_\Delta$ associated to $L$ with the $\Delta$ integral is defined by
\begin{equation}
\mathcal{L}_{\Delta}(X)=\left [ J_{\Delta} (L(T,X,\Delta X)) \right ]_N .
\end{equation}
\end{definition}

The previous form is completely fixed once one has given a discrete embedding formalism. We clearly see the relation between the classical function and the discrete one. 

\subsubsection{Comparison with Marsden-West definition}

In (\cite{mars} p.363 and $\S$.1.3 p.370-371), the authors define discrete Lagrangian system for which they do not preserve the classical form of the continuous functional. They introduce a discrete Lagrangian given by 
\begin{equation}
\label{discrete-lagrange-marsden}
L_d (X_k ,X_{k+1} , h ) = h L (X_k , \Delta (X)_k , t_k ) ,
\end{equation}
which gives the discrete functional 
\begin{equation}
\label{lagrange-marsden}
\mathcal{L}_h (X)=\di\sum_{k=0}^{N-1} L_d (X_k ,X_{k+1} ,h) .
\end{equation}
The algebraic structure of the classical functional is then destroyed. This point has an important consequence in the derivation of the discrete Euler-Lagrange equation in (\cite{mars},p.363 and $\S$.1.3, Theorem 1.3.1 p.371). Indeed, as we will see, the form obtained in \cite{mars} does not put in evidence the complete analogy between the discrete Euler-Lagrange equation and the classical one. 

The same remark applies to the presentation made by E. Hairer, C. Lubich and G. Wanner in the book {\it geometric numerical integration} (see \cite{lubi}, formula (6.5) and (6.6) p.192-195).

\subsubsection{Discrete calculus of variations and discrete Euler-Lagrange equations}

An element of $C_0(\TT,\R)$ is called a {\it discrete variation}. As $C(\TT,\R)$ is a linear space, we can define the Frechet derivative of $\mathcal{L}_{\Delta}$ (resp. $\mathcal{L}_{\nabla}$) along a given direction $H \in C(\TT,\R)$ and denoted by
\begin{equation}
\mathcal{D}\mathcal{L}_{\Delta}(X)(H)=\lim_{\epsilon\rightarrow 0}\frac{1}{\epsilon}(\mathcal{L}_{\Delta}(\mathrm{X}+\epsilon H) -\mathcal{L}_{\Delta}(H)) .
\end{equation}
The corresponding notion of critical points is given by:

\begin{definition}
A discrete critical point $X\in C(\TT,\R)$ verify $\mathcal{D}\mathcal{L}_{\Delta}(X)(H)=0$ for all $H\in C_0(\TT,\R)$.
\end{definition}

We obtain the following {\it discrete Euler-Lagrange equation} : 

\begin{theorem}[Discrete Euler-Lagrange equation]
\label{discrete-euler}
Let $L$ be an admissible Lagrangian function. A discrete function $X\in C(\TT,\R)$ is a critical point of the discrete $\Delta$ Lagrangian functional associated to $L$ if and only if it satisfies 
\begin{equation}
\nabla \left( \frac{\partial L}{\partial v}(T,X,\Delta X)\right) = \frac{\partial L}{\partial x}(T,X,\Delta X) , \end{equation}
over $\TT^\pm$.
\end{theorem}

The previous formulation makes clear the relation between the usual Euler-Lagrange equation and the discrete one. In particular, one see the duality between the two operators $\Delta$ and $\nabla$ which induces a naturak mixing between the two discrete derivatives which is hide in the continuous case. 

\subsubsection{Comparison with Marsden-West definition}

One can compare this writing of the discrete Euler-Lagrange equation with the one obtained in (\cite{mars},p.363). Using the form (\ref{lagrange-marsden}), they have 
\begin{equation}
\label{euler-lagrange-marsden}
D_2 L_d (X_{k-1} ,X_k , h) +D_1 L_d (X_k , X_{k+1} ,h ) =0,
\end{equation}
for $k=1,\dots ,N-1$. The usual form of the Euler-Lagrange equation is completely lost and by the way the analogy between the two objects.\\

The same remark applies to the presentation made by E. Hairer, C. Lubich and G. Wanner in the book {\it geometric numerical integration} (see \cite{lubi}, formula (6.7) p.192-195).

\subsubsection{Proof of Theorem \ref{discrete-euler}}

The proof can be done in different ways. However, we want to keep a proof which is similar to the classical proof of the continuous Euler-Lagrange equation as exposed for example in (\cite{gelfand}, Theorem 1,p.15).\\

Using a Taylor expansion of $L$, we obtain 
\begin{equation}
\mathcal{D}\mathcal{L}_\Delta(X)(H)=\left [ J_{\Delta}\left(\frac{\partial L}{\partial v}(T,X,\Delta X)\star \Delta H+\frac{\partial L}{\partial x}(T,X,\Delta X)\star H\right) \right ]_N.
\end{equation}
As $H\in C_0(\TT,\R)$, we have using the discrete integration by parts formula
\begin{equation}
\mathcal{D}\mathcal{L}_\Delta(X)(H)=\left [ J_{\Delta} 
\left ( 
-\nabla \left(\frac{\partial L}{\partial v}(T,X,\Delta X)\right) \star H 
+\frac{\partial L}{\partial x}(T,X,\Delta X)\star H
\right) 
\right ]_N \ .
\end{equation}
Using the discrete Dubois-Raymond lemma, we deduce 
\begin{equation}
-\nabla \left[ \frac{\partial L}{\partial v}(T,X,\Delta X) \right ] +\frac{\partial L}{\partial x}(T,X,\Delta X) =0 \ \ \mbox{\rm over}\ \ \TT^{\pm}.
\end{equation}
This concludes the proof.

\part{Higher order discrete embeddings}

As it is well-known there exist other way to obtain better approximation of discrete integral for numerical computations such as the midpoint method, the trapezoidal method and the Simpson method. Our choice of discrete integral correspond to the left method for the $\Delta$ integral and the right method for the $\nabla$ integral. In our construction of finite difference embedding, the choice of discrete integral is explicitly dependent of the choice of the discretisation of functions and the discrete derivative in order to preserve the fundamental theorem of differential calculus. In what follow we briefly present two higher interpolation method which leads to the midpoint method, the trapezoidal method and the Simpson method for the forward derivative only. The backward derivative will be straightforward obtained.

\section{Quadratic embedding : Trapezoidal and Midpoint methods}

In this Section, we obtain the classical one step Trapezoidal (resp. Midpoint) numerical scheme (see \cite{dema}, III.1.2) as respectively the differential (resp. integral) discrete embedding of the quadratic discrete embedding that we define in the next Section.

\subsection{Quadratic discrete embedding}

We follow the general strategy of discrete embeddings. We define a quadratic interpolation map which is used to obtain a discrete derivative denoted by $\Delta_2$. We then construct the corresponding discrete integral denoted by $J_{\Delta_2}$. 

\subsubsection{Quadratic interpolation mapping}

We suppose $N=2d,d \ge 1, d \in \N$. We define a {\it quadratic interpolation mapping} denoted by $\iota_2 : C(\TT ,\R ) \rightarrow P_2 ([a,b],\R )$ using Lagrange polynomial of degree $2$. Precisely, we have 
\begin{equation}
\iota_2 (F) (t) =\di\sum_{i=0}^{d-1} P_{L,2}^{F_{2i} ,F_{2i+1} ,F_{2i+2}} (t) 1_{[t_{2i} ,t_{2i+2}[} . 
\end{equation}

\begin{figure}[ht!]
\centering
\includegraphics[width=0.4\textwidth,clip]{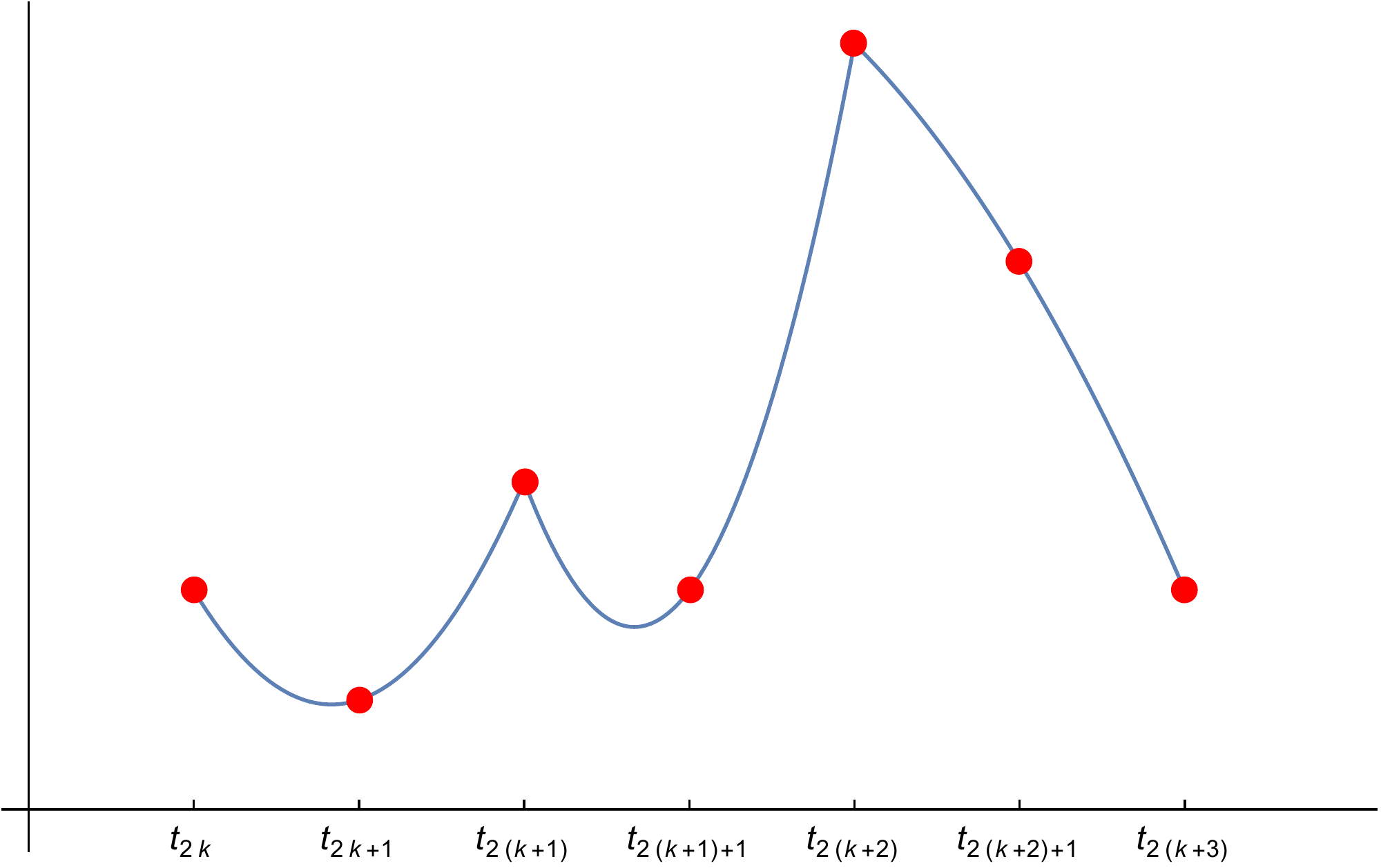}      
\caption{Quadratic interpolation mapping}
\label{p2}
\end{figure}

\subsubsection{Quadratic discrete derivative and integral}

We follow the strategy used in Part I. We denote by $\Delta_2$ the discrete derivative defined by 
\begin{equation}
\Delta_2 =\pi \circ d^+ \circ \iota_2 .
\end{equation}

An explicit form is given by
\begin{equation}
\left .
\begin{array}{lll}
\left [ \Delta_2 (F)\right ]_{2k} & = & \di\frac{2(F_{2k+1}-F_{2k})}{h}-\frac{(F_{2k+2}-F_{2k})}{2h}, \\
\left [ \Delta_2 (F) \right ]_{2k+1} & = & \di\frac{(F_{2k+2}-F_{2k})}{2h},
\end{array}
\right .
\end{equation}
for $k=0,...,d-1$.

\begin{figure}[ht!]
\centering
\includegraphics[width=0.4\textwidth,clip]{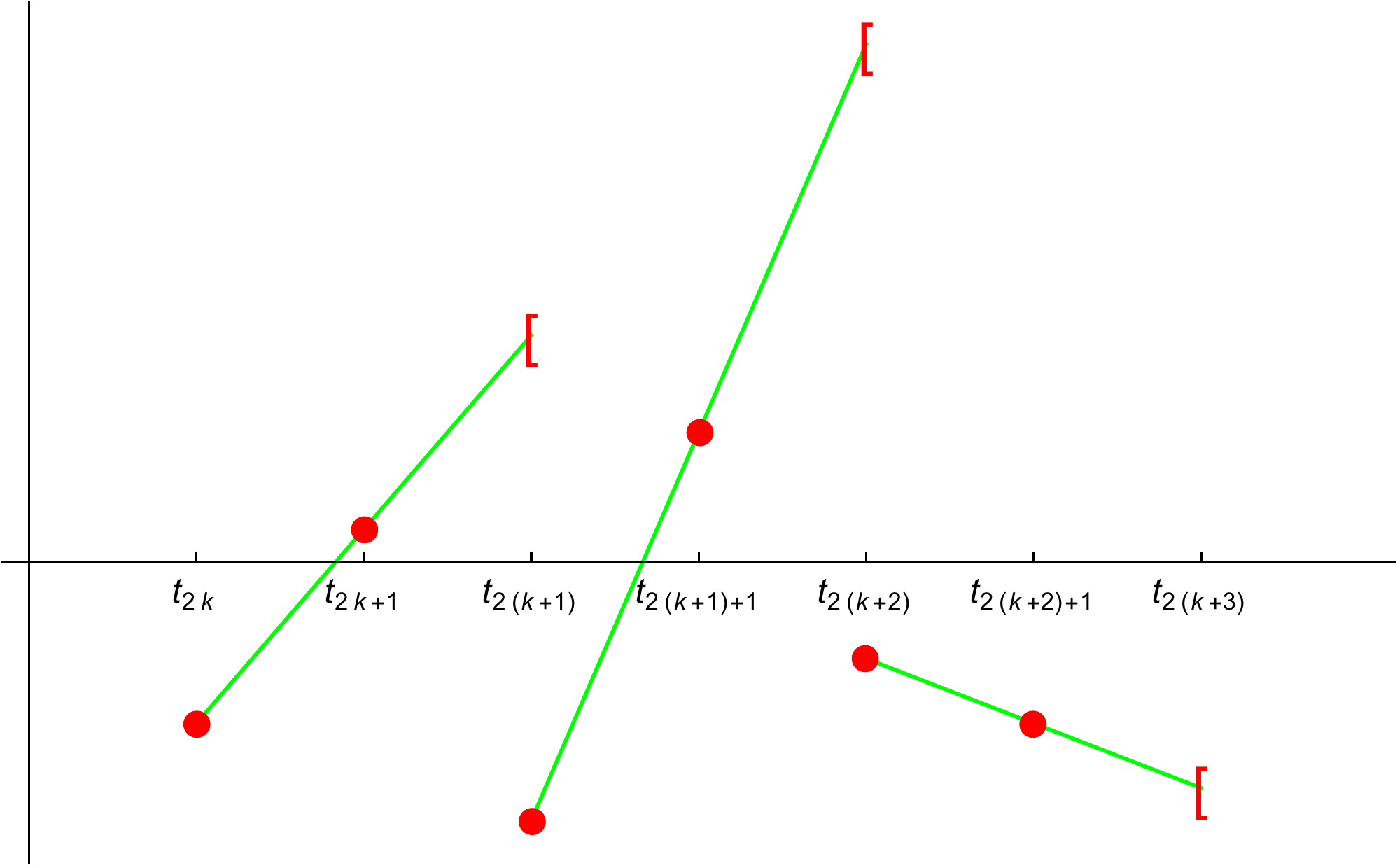}      
\caption{Quadratic forward derivative}
\label{dp2}
\end{figure}

In order to define the discrete integral associated to $\Delta_2$, we must find an interpolation mapping denoted by $\iota_1^+$ in the following, such that $\iota_1^+ \circ \pi =\mbox{\rm Id}$ on $P_1^+ ([a,b])$. We have 
\begin{equation}
\iota_1^+(F)(t)=\di\sum_{k=0}^{d-1} \left[ F_{2k} + \frac{F_{2k+1}-F_{2k}}{h}(t-t_{2k})\right ] 1_{[t_2k,t_{2k+2}[}(t).
\end{equation}
The discrete integral is then defined by 
\begin{equation}
J_{\Delta_2} =\pi \circ \int_a^t \circ \iota_1^+ ,
\end{equation}
and satisfies $J_{\Delta_2} (\Delta_2 (F)) =F-\mathbb{F}_0$. Indeed, we have $J_{\Delta_2} (\Delta_2 (F)) =
\pi \circ \int_a^t \circ \tau_1^+ \circ \pi \circ d^+ \circ \iota_2 (F)$. As $d^+ \circ \iota_2 (F) \in P_1^+ ([a,b])$ and $\tau_1^+ \circ \pi =\mbox{rm Id}$ on $P_1^+ ([a,b])$, this equality reduces to $J_{\Delta_2} (\Delta_2 (F)) =
\pi \circ \int_a^t \circ d^+ \circ \iota_2 (F) =\pi \left [ \iota_2 (F) -F_0 \right ] =F -\mathbb{F}_0$.\\

An explicit form for $J_{\Delta_2}$ is given by 
\begin{equation}
\left [ J_{\Delta} (F) \right ] _{2k} =\sum_{q=0}^{k-1} 2h F_{2q+1},
\end{equation}
for $k=1,...,d$ with $[ J_{\Delta_2} (F)] _0 = 0$ and 
\begin{equation}
\left [ J_{\Delta_2} (F) \right ] _{2k+1} = h\frac{F_{2k}+F_{2k+1}}{2} + \sum_{q=0}^{k-1}2h F_{2q+1},
\end{equation}
for $k=1,...,d-1$ with $[J_{\Delta_2} (F)] _1 = h\frac{F_{0}+F_{1}}{2}$.

\begin{figure}[ht!]
\centering
\includegraphics[width=0.4\textwidth,clip]{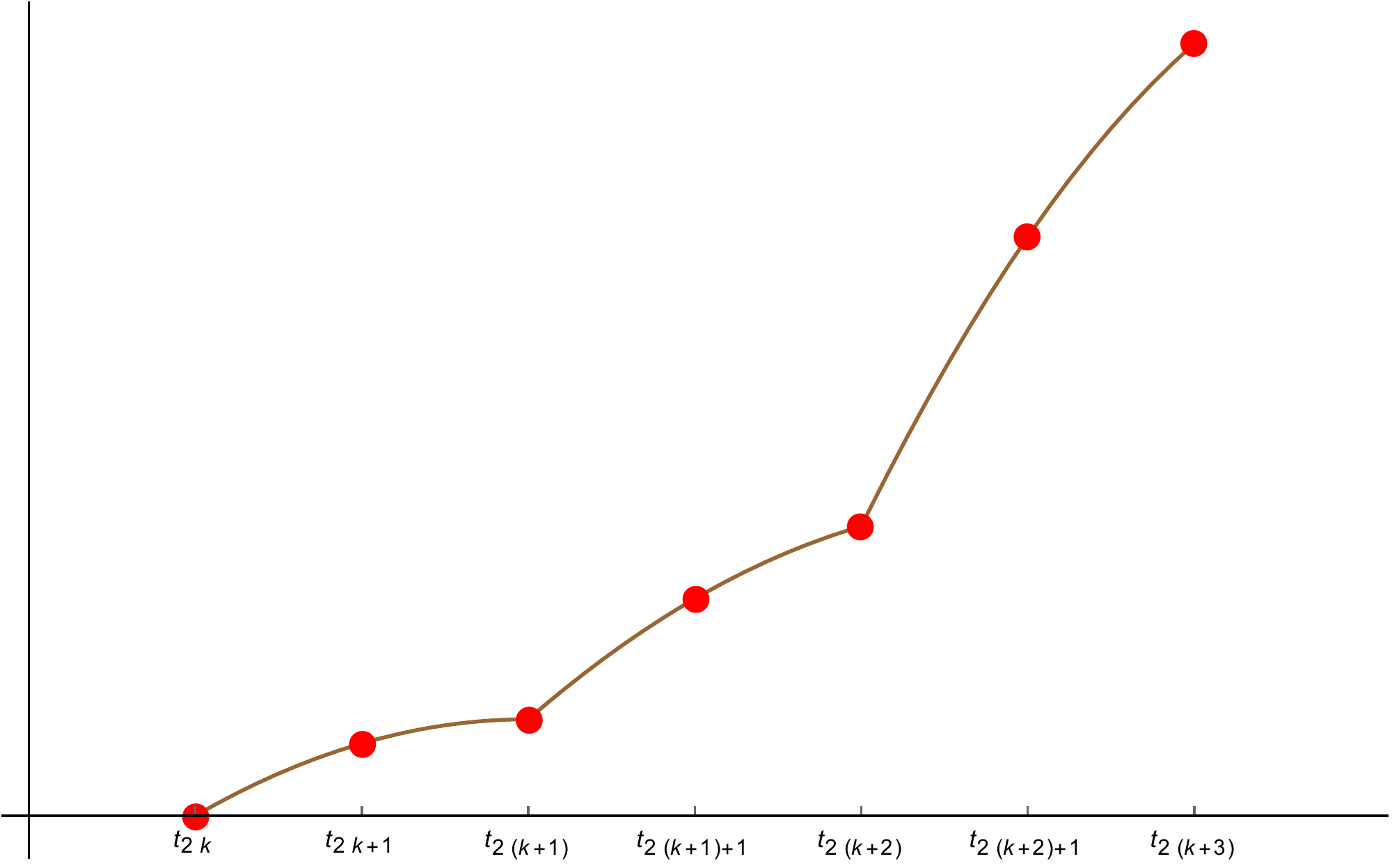}      
\caption{Quadratic discrete integral}
\label{intp2}
\end{figure}

\subsection{Discrete differential and integral embedding}

Applying the previous discrete embedding, we obtain the following result :

\begin{theorem}
\label{delta2}
The $\Delta_2$ differential embedding (resp. integral embedding) correspond to the one step Trapezoidal (resp. Midpoint) scheme. 
\end{theorem}

A direct consequence is the following corollary :

\begin{corollary}
The $\Delta_2$ discrete embedding is not coherent.
\end{corollary}

It is interesting to notice that the $\Delta$ discrete embedding was coherent with respect to the differential/integral embedding only breaking when one considers variational discrete embedding. In this case, non-coherence already break-up at the first level.

\section{Cubic discrete embedding : the forward Simpson numerical scheme}

In this Section, we consider a cubic discrete embedding. The main point is that we recover coherence between the differential/integral embedding and obtain the {\it forward Simpson numerical scheme} (see \cite{dema},III.1.2).

\subsection{Cubic discrete embedding}

\subsubsection{Cubic interpolation mapping}

We suppose $N=3d,d \ge 1, d \in \N$. We define a {\it cubic interpolation mapping} denoted by $\iota_3 : C(\TT ,\R ) \rightarrow P_3 ([a,b],\R )$ using Lagrange polynomial of degree $3$. Precisely, we have 
\begin{equation}
\iota_3 (F) (t) =\di\sum_{i=0}^{d-1} P_{L,3}^{F_{3i} ,F_{3i+1} ,F_{3i+2} ,F_{3i+3} } (t) 1_{[t_{3i} ,t_{3i+3}[} . 
\end{equation}

\begin{figure}[ht!]
\centering
\includegraphics[width=0.4\textwidth,clip]{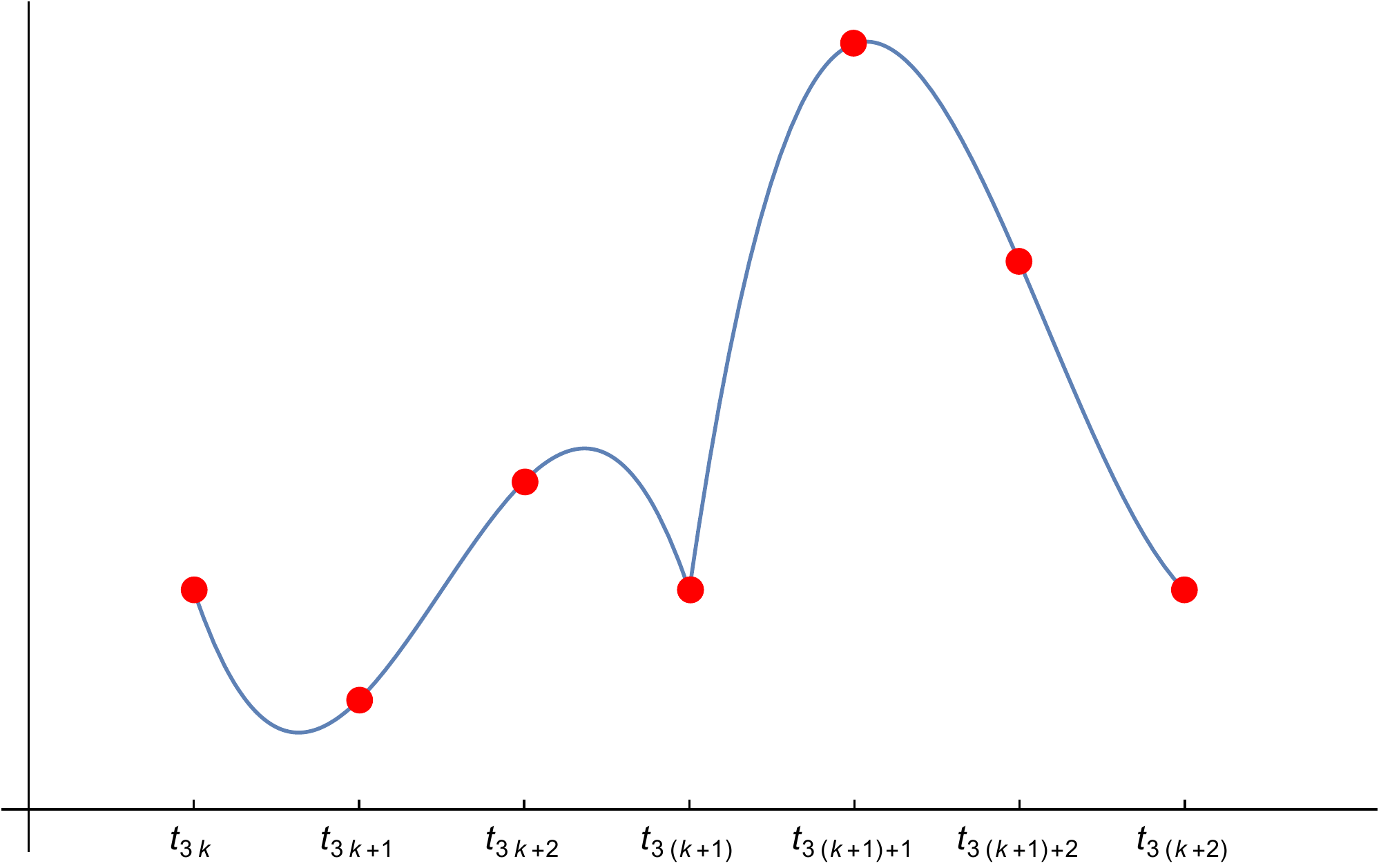}      
\caption{Cubic interpolation mapping}
\label{p3}
\end{figure}

\subsubsection{Cubic discrete derivative and integral}

We follow the strategy used in Part I. We denote by $\Delta_3$ the discrete derivative defined by 
\begin{equation}
\Delta_3 =\pi \circ d^+ \circ \iota_3 .
\end{equation}

An explicit form for $\Delta_3$ is given by 
\begin{equation}
\left .
\begin{array}{lll}
\left [ \Delta_3 (F)\right ]_{3k} & = & -3\di\frac{(F_{3k+2}-2F_{3k+1}+F_{3k})}{2h}-\frac{(F_{3k+3}-F_{3k})}{3h}, \\
\left [\Delta_3 (F) \right ]_{3k+1}& = & -\di\frac{(F_{3k+1}-2F_{3k+2}+F_{3k})}{2h}-\frac{(F_{3k+3}-F_{3k})}{6h}, \\
\left [ \Delta_3 (F) \right ]_{3k+2}& = & \di\frac{(F_{3k+3}-2F_{3k+1}+F_{3k+2})}{2h}-\frac{(F_{3k+3}-F_{3k})}{6h},
\end{array}
\right .
\end{equation}
for $k=0,...,d-1$.

\begin{figure}[ht!]
\centering
\includegraphics[width=0.4\textwidth,clip]{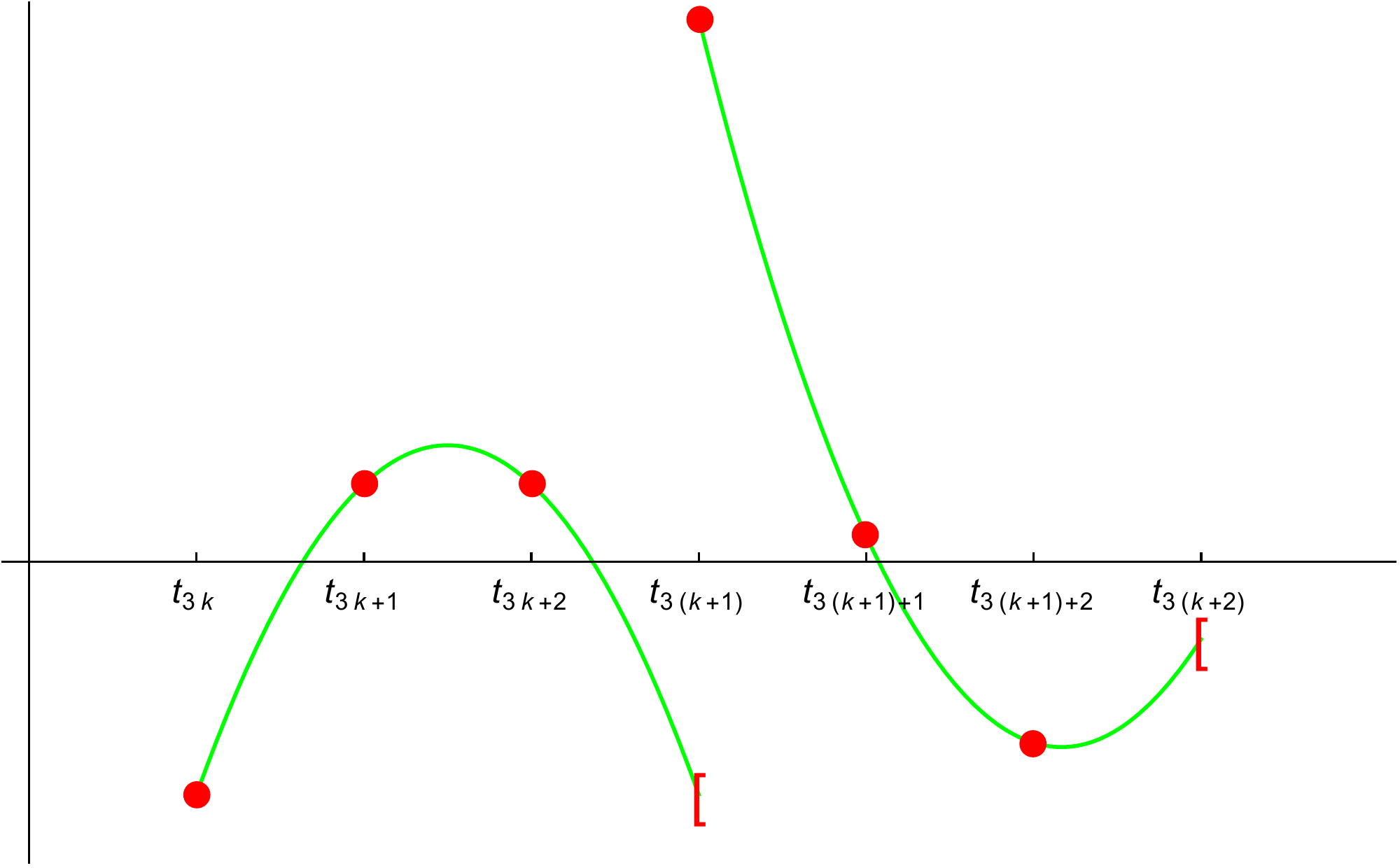}      
\caption{Cubic forward derivative}
\label{dp3}
\end{figure}

In order to define the discrete integral associated to $\Delta_3$, we must find an interpolation mapping denoted by $\iota_2^+$ in the following, such that $\iota_2^+ \circ \pi =\mbox{\rm Id}$ on $P_2^+ ([a,b])$. We have 
\begin{equation}
\iota_2^+(F)(t)=\sum_{k=0}^{d-1} 
\left [ 
\begin{array}{l}
\di \frac{F_{3k} (t-t_{3k+1}) (t-t_{3k+2})}{(t_{3k}-t_{3k+1}) (t_{3k}-t_{3k+2})}-\frac{ F_{3k+1} (t-t_{3k}) (t-t_{3k+2})}{(t_{3k}-t_{3k+1}) (t_{3k+1}-t_{3k+2})} \\
+\di\frac{F_{3k+2} (t-t_{3k}) (t-t_{3k+1})}{(t_{3k+2}-t_{3k}) (t_{3k+2}-t_{3k+1})} 
\end{array}
\right ] 
1_{[t_{3k},t_{3k+3}[}(t).
\end{equation}

The discrete integral is then defined by 
\begin{equation}
J_{\Delta_3} =\pi \circ \int_a^t \circ \iota_2^+ ,
\end{equation}
and satisfies $J_{\Delta_3} (\Delta_3 (F)) =F-\mathbb{F}_0$. Indeed, we have $J_{\Delta_3} (\Delta_3 (F)) =
\pi \circ \int_a^t \circ \iota_2^+ \circ \pi \circ d^+ \circ \iota_3 (F)$. As $d^+ \circ \iota_3 (F) \in P_2^+ ([a,b])$ and $\iota_1^+ \circ \pi =\mbox{rm Id}$ on $P_2^+ ([a,b])$, this equality reduces to $J_{\Delta_2} (\Delta_2 (F)) =
\pi \circ \int_a^t \circ d^+ \circ \iota_3 (F) =\pi \left [ \iota_3^+ (F) -F_0 \right ] =F -\mathbb{F}_0$.\\

An explicit form for $J_{\Delta_3}$ is given by

\begin{equation}
\left [ J_{\Delta_3} (F) \right ]_{3k} = \sum_{q=0}^{k-1}3h \frac{\left(F_{3k} + 3F_{3k+2}\right)}{4},
\end{equation}
for $k=1,...,d$ and by
\begin{align*}
\left [ J_{\Delta_3} (F) \right ]_{3k+1} =\sum_{q=0}^{k-1}\frac{h}{12}\left(5 F_{3k} + 8 F_{3k+1}-F_{3k+2}\right) \\
\left [ J_{\Delta_3} (F) \right ]_{3k+2} = \sum_{q=0}^{k-1}\frac{h}{3}\left(F_{3k} + 4F_{3k+1}+F_{3k+2}\right),
\end{align*}
for $k=1,...,d-1$.

\begin{figure}[ht!]
\centering
\includegraphics[width=0.4\textwidth,clip]{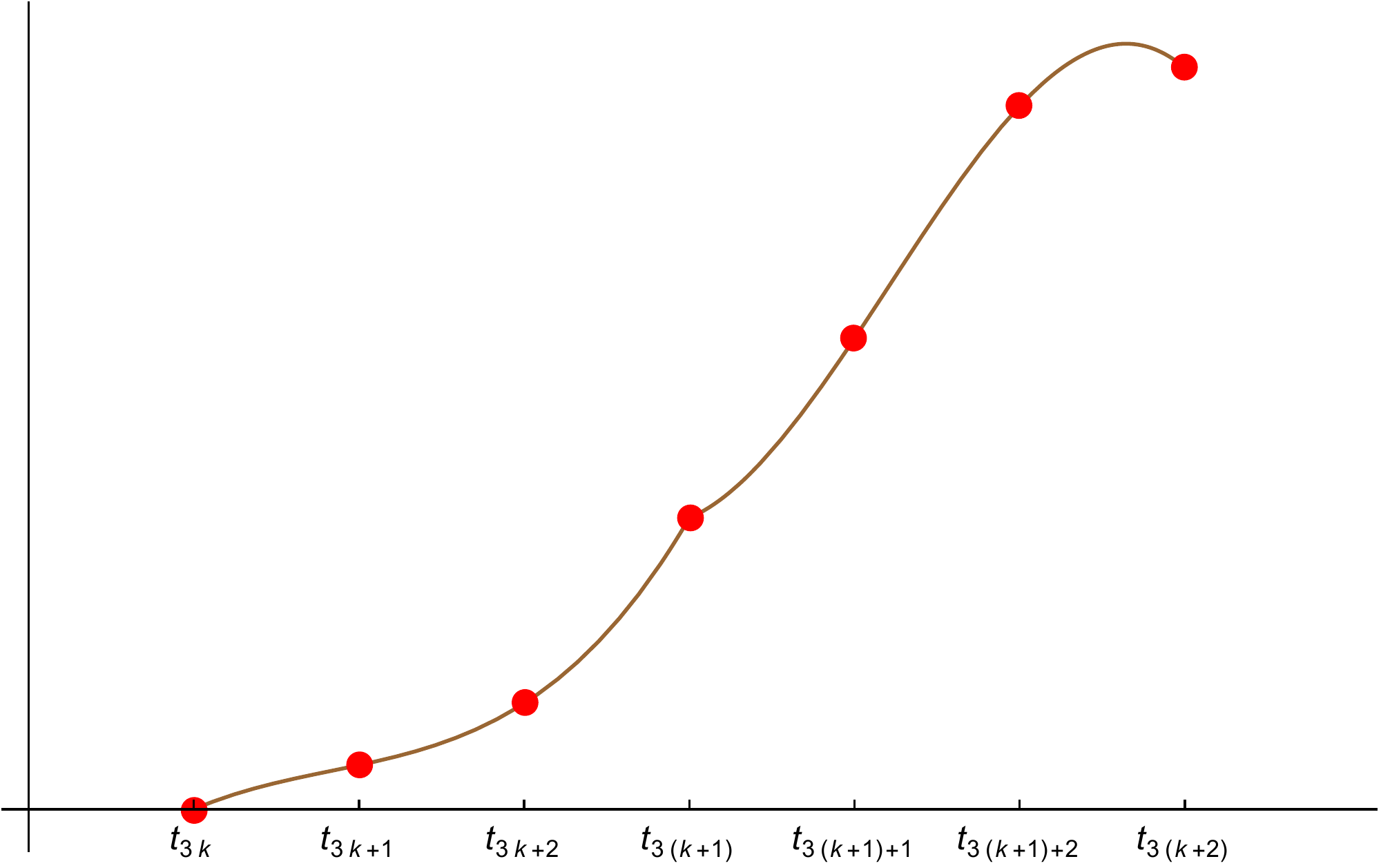}      
\caption{Cubic discrete integral}
\label{intp3}
\end{figure}

\subsection{Discrete differential and integral embedding}

We have the following Theorem :

\begin{theorem}
\label{delta3}
The $\Delta_3$ differential and integral discrete embedding corresponds to the one forward Simpson numerical scheme.
\end{theorem}

A direct consequence of this Theorem is :

\begin{corollary}
The $\Delta_3$ discrete embedding is coherent.
\end{corollary}

\section{Conclusion and perspectives}

The previous formalism includes the classical numerical methods in a unified point of view called {\it discrete embedding formalisms} which is itself an illustrative example of the philosophy of {\it embedding formalisms}. Doing so, we obtain others formulations of classical results like the discrete calculus of variations, the discrete Euler-Lagrange equation, discrete integration by part, etc which allows us to clearly put in evidence the correspondence between the continuous case and the discrete case. As explained in the last Section, this procedure is not limited to first order methods but can be used for arbitrary high order. Moreover, the strategy can also be adapted in order to consider some particular equations based on very specific differential operators. An interesting problem with respect to our approach is to give a full categorical formulation of discrete embedding formalisms. 

\begin{appendix}

\section{Proof of Theorem \ref{delta2}}

The proof is based on explicit computations of the different quantities. \\

The discrete $\Delta_2$ differential embedding of (\ref{eqdiff}) is given by $\Delta_2 X = f(T,X)$ which reduces to
\begin{align*}
\frac{2(X_{2k+1}-X_{2k})}{h}-\frac{(X_{2k+2}-X_{2k})}{2h} &=f(T_{2k},X_{2k}), \\
\frac{(X_{2k+2}-X_{2k})}{2h} &=f(T_{2k+1},X_{2k+1}),
\end{align*}
for $k=0,...,d-1$.\\

The discrete $\Delta_2$ integral embedding of (\ref{eqint}) is given by $X=\mathbb{X}_0 +J_{\Delta_2} (f(T,X))$ which gives
\begin{equation}
X_{2k} = X_0 + \sum_{q=0}^{k-1} 2h f(T_{2q+1},X_{2q+1}),
\end{equation}
for $k=1,...,d$ and 
\begin{equation}
X_{2k+1} = X_0 + h\frac{f(T_{2k},X_{2k})+f(T_{2k+1},X_{2k+1})}{2} + 2h\sum_{q=0}^{k-1}f(T_{2q+1},X_{2q+1}),
\end{equation}
for $k=0,...,d-1$.\\

As a consequence, the discrete $\Delta_2$-differential embedding is equivalent to
\begin{align*}
X_{2k+1} &=X_{2k} + h \frac{f(T_{2k},X_{2k})+f(T_{2k+1},X_{2k+1})}{2},\ \mbox{\rm for} k=0,...,d-1,
\end{align*}
which is the classical one step {\it trapezoidal scheme} in Numerical Analysis (see \cite{dema},III.1.2 (c) p.63)and also
\begin{align*}
X_{2k+2} &=X_{2k} + 2hf(T_{2k+1},X_{2k+1})
\end{align*}
which is the classical one step {\it forward midpoint scheme} in Numerical Analysis (see \cite{dema},III.1.2 (a) p.61). This concludes the proof.

\section{Proof of Theorem \ref{delta3}}

The proof of Theorem \ref{delta3}  is based on explicit computations. The discrete $\Delta_3$ differential embedding of (\ref{eqdiff}) is given by $\Delta_3 X = f(T,X)$ which reduces to
\begin{align*}
-3\frac{(F_{3k+2}-2F_{3k+1}+F_{3k})}{2h}-\frac{(F_{3k+3}-F_{3k})}{3h}&=f(T_{3k},X_{3k}), \\
-\frac{(F_{3k+1}-2F_{3k+2}+F_{3k})}{2h}-\frac{(F_{3k+3}-F_{3k})}{6h}&=f(T_{3k+1},X_{3k+1}), \\
\frac{(F_{3k+3}-2F_{3k+1}+F_{3k+2})}{2h}-\frac{(F_{3k+3}-F_{3k})}{6h}&=f(T_{3k+2},X_{3k+2}),
\end{align*}
for $k=0,...,d-1$.\\

The discrete $\Delta_3$ integral embedding of (\ref{eqint}) is given by $X=\mathbb{X}_0 +J_{\Delta_3} (f(T,X))$ which gives
\begin{equation}
X_{3k} = X_0 + \sum_{q=0}^{k-1}3h \frac{\left(f(T_{3k},X_{3k}) + 3f(T_{3k+2},X_{3k+2})\right)}{4},
\end{equation}
for $k=1,...,d$ and
\begin{align*}
X_{3k+1} &= X_0 + \sum_{q=0}^{k-1}\frac{h}{12}\left(5 f(T_{3k},X_{3k}) + 8 f(T_{3k+1},X_{3k+1})-f(T_{3k+2},X_{3k+2})\right), \\
X_{3k+2} &= X_0 + \sum_{q=0}^{k-1}\frac{h}{3}\left(f(T_{3k},X_{3k})+4f(T_{3k+1},X_{3k+1})+f(T_{3k+2},X_{3k+2})\right),
\end{align*}
for $k=0,...,d-1$.\\

As a consequence, we obtain for $k=0,...,d-1$
\begin{align*}
X_{3k+2} &=X_{3k} + \frac{h}{3}\left(f(T_{3k},X_{3k})+4f(T_{3k+1},X_{3k+1})+f(T_{3k+2},X_{3k+2})\right),
\end{align*}
which is the classical {\it forward Simpson scheme} in Numerical Analysis (see \cite{dema}, III.1.2 (c) p.63). This concludes the proof.
\end{appendix}

\bibliographystyle{plain}

\end{document}